\documentclass[12pt]{article}

\usepackage{amssymb}
\usepackage{amsmath}
\usepackage{amsfonts}
\usepackage{setspace}
\usepackage{epsfig}
\usepackage{subfigure}
\usepackage{color}
\usepackage{enumerate}

\numberwithin{equation}{section}

\begin{document}
\baselineskip=22pt

\newcommand{\ve}[1]{\mbox{\boldmath$#1$}}
\newcommand{\be}{\begin{equation}}
\newcommand{\ee}{\end{equation}}
\newcommand{\ben}{\begin{equation*}}
\newcommand{\een}{\end{equation*}}
\newcommand{\bc}{\begin{center}}
\newcommand{\ec}{\end{center}}
\newcommand{\bal}{\begin{align*}}
\newcommand{\enal}{\end{align*}}
\newcommand{\al}{\alpha}
\newcommand{\bt}{\beta}
\newcommand{\gm}{\gamma}
\newcommand{\de}{\delta}
\newcommand{\la}{\lambda}
\newcommand{\om}{\omega}
\newcommand{\Om}{\Omega}
\newcommand{\Gm}{\Gamma}
\newcommand{\De}{\Delta}
\newcommand{\Th}{\Theta}
\newcommand{\kp}{\kappa}
\newcommand{\nno}{\nonumber}
\newtheorem{theorem}{Theorem}[section]
\newtheorem{lemma}{Lemma}[section]
\newtheorem{assum}{Assumption}[section]
\newtheorem{claim}{Claim}[section]
\newtheorem{proposition}{Proposition}[section]
\newtheorem{corollary}{Corollary}[section]
\newtheorem{definition}{Definition}[section]
\newtheorem{remark}{Remark}[section]
\newenvironment{proof}[1][Proof]{\begin{trivlist}
\item[\hskip \labelsep {\bfseries #1}]}{\end{trivlist}}
\newenvironment{proofclaim}[1][Proof of Claim]{\begin{trivlist}
\item[\hskip \labelsep {\bfseries #1}]}{\end{trivlist}}

\def \qed {\hfill \vrule height7pt width 5pt depth 0pt}
\def\refhg{\hangindent=20pt\hangafter=1}
\def\refmark{\par\vskip 2.50mm\noindent\refhg}

\title{\textbf{Convergence of Cubic Spline Super Fractal Interpolation Functions}}
\date{}
\author{G.P. Kapoor$^1$  and Srijanani Anurag Prasad$^2$}
\maketitle
\vspace{-1.5cm}
\bc Department of Mathematics and Statistics\\
Indian Institute of Technology Kanpur\\  Kanpur 208016 India\\
$^1$gp@iitk.ac.in    $^2$jana@iitk.ac.in\ec

\begin{abstract}
In the present work, the notion of Cubic Spline Super Fractal
Interpolation Function (SFIF) is introduced to simulate an object
that depicts one structure embedded into another and its
approximation properties are investigated. It is shown that, for an
equidistant partition points of $[x_0,x_N]$, the interpolating Cubic
Spline \textit{SFIF} $g_{\sigma}(x) \equiv g_{\sigma}^{(0)}(x)$ and
their derivatives $g_{\sigma}^{(j)}(x)$ converge respectively to the
data generating function $y(x) \equiv y^{(0)}(x)$ and its
derivatives $y^{(j)}(x)$ at the rate of $h^{2-j+\epsilon}
(0<\epsilon<1) , j=0,1,2, $ as the norm $h$ of the partition of
$[x_0,x_N]$ approaches zero. The convergence results for Cubic
Spline \textit{SFIF} found here show that any desired accuracy can
be achieved in the approximation of a regular data generating
function and its derivatives by a Cubic Spline \textit{SFIF} and its
corresponding derivatives.
\end{abstract}



\newpage
\doublespacing

\section{Introduction}

Barnsley~\cite{barnsley86}
introduced Fractal Interpolation Function (FIF) using the theory of Iterated Function System (IFS).  Later, Barnsley~\cite{barnsley05,barnsley06,barnsley08} introduced the class of super fractal sets constructed by using multiple IFSs to simulate such
objects. Massopust~\cite{massopust10} constructed super fractal functions and V-variable fractal functions by joining pieces of fractal functions which are attractor of finite family of IFss.

FIF, constructed as attractor of a single Iterated Function System (IFS) by virtue of
self-similarity alone, is not rich enough to describe an object
found in nature or output of a certain scientific experiment. The
objects of nature generally reveal one or more structures embedded
in to another. Similarly, the outcomes of several scientific
experiments exhibit randomness and variation at various stages.
Therefore, more than one IFSs are needed to model such objects. A
solution of fractal interpolation problem based on several IFS to
model such objects is introduced in~\cite{kapoor_c2} by introducing
the notion of \textit{Super Fractal Interpolation Function~(SFIF)}.
The construction of SFIF  use more than one IFS wherein, at each
level of iteration, an IFS is chosen from a pool of several IFS.
This approach ensured desired randomness and variability needed to
facilitate better geometrical modeling of objects found in nature
and results of certain scientific experiments.

Spline functions, introduced by Schoenberg~\cite{schoenberg46}, find
vast applications in areas like data fitting~\cite{knott00},
computer aided geometric design~\cite{bartels89,farin90}, numerical
solutions of differential equations~\cite{micula99}, etc. A
piecewise polynomial function  $\vartheta$ on an interval
$[x_0,x_N]$, which is composed of subintervals $[x_{i-1},x_i]$,
$i=1,2,\ldots,N$, is called a \emph{Spline} of order $n$   if (i)
$\vartheta(x)$ is a continuous polynomial of degree atmost $n-1$ in
each subintervals $[x_{i-1},x_i]$, $i=1,2,\ldots,N$,  and (ii) the
derivatives  $\vartheta^{(m)}$, $0 \leq m \leq n-2$, are continuous
on $[x_0,x_N]$. A \emph{Cubic Spline} is a Spline of degree $3$. For
a data set $\{x_i\} $ of $n+1$ points, a Cubic Spline is constructed
with $n$ piecewise cubic polynomials between the data points. If
$\vartheta$ represents a Cubic Spline approximating the
function $y \in C^4[x_0,x_N]$, then $\vartheta$ is twice
continuously differentiable and $\vartheta(x_i)=y(x_i)$.

Navascues and Sebastian~\cite{navascues03} considered a Cubic Spline FIF as a
generalization of classical Spline and obtained estimates on error
in approximation of the data generating function by a Cubic Spline
FIF. However, their Cubic Spline FIF was constructed using a single
IFS and so it is not equipped enough to simulate an object that
depicts one structure embedded into another. To approximate an
object by a spline-like FIF, the concept of Cubic Spline
\textit{SFIF} is introduced in the present work and its
approximation properties are investigated.  The convergence results
for Cubic Spline \textit{SFIF} found here show that any desired
accuracy can be achieved in the approximation of a regular data
generating function and its derivatives by a Cubic Spline
\textit{SFIF} and its corresponding derivatives.

The organization of the present chapter is as follows: In
Section~\ref{sec:sfif}, a brief review on the construction of  Super
Fractal Interpolation Function for a given finite set of data is
given. The notion of Cubic Spline \textit{SFIF} is introduced in
Section~\ref{sec:convsfif}. It is proved in this section that, for
an equidistant partition points of $[x_0,x_N]$, the interpolating
Cubic Spline \textit{SFIF} $g_{\sigma}(x) \equiv
g_{\sigma}^{(0)}(x)$ and their derivatives $g_{\sigma}^{(j)}(x)$
converge respectively to the data generating function $y(x) \equiv
y^{(0)}(x)$ and its derivatives $y^{(j)}(x)$ at the rate of
$h^{2-j+\epsilon} (0<\epsilon<1) , j=0,1,2, $ as the norm $h$ of the
partition of $[x_0,x_N]$ approaches zero.

\section{Construction of SFIF}\label{sec:sfif}

In this section,  a brief introduction on the construction of Super
Fractal Interpolation Function (SFIF) is given.

Let $ S_0 = \{ (x_i,y_i) \in \mathbb{R}^2 :i=0,1,\ldots N \}$ be the
set of given interpolation data. The contractive homeomorphisms $
L_n : I \rightarrow I_n $  for $ n = 1,\ldots, N $, are  defined by
\begin{align}\label{eq:Ln}
L_n(x) = a_n x + b_n
\end{align}
where, $ a_n = \frac{x_n-x_{n-1}}{x_N-x_0}$ and $ b_n = \frac{x_N
x_{n-1}-x_0 x_n}{x_N-x_0} $. For $k=1,2,\ldots,M$, $M>1$ and $ n \ =
\ 1,2,\ldots,N $, let the functions $G_{n,k} : I \times \mathbb{R}
\rightarrow \mathbb{R}$  defined by
\begin{align}\label{eq:gnk}
G_{n,k}(x,y)=  \gm_{n,k} y +  q_{n,k}(x)
\end{align}
satisfy the join-up conditions
\begin{align}\label{eq:jc}
G_{n,k}(x_0, y_{0})& =  y_{n-1} \quad \mbox{and} \quad G_{n,k}(x_N,
y_{N}) = y_n.
\end{align}
Here, $\gm_{n,k}$ are free parameters chosen such that $|\gm_{n,k}|
< 1$ and $\gm_{n,k} \neq \gm_{n,l}$ for $k \neq l$.

The Super Iterated Function System (SIFS) that is needed to
construct SFIF corresponding to the set of given interpolation data
$S_0$ is defined as the pool of IFS
\begin{align}\label{eq:sifs}
\Big\{ \big\{\mathbb{R}^2; \om_{n,k} : n =1,2,\ldots,N \big\}, \ k=
1,2,\ldots,M \Big\}
\end{align}
where, the functions $\ \om_{n,k} :I \times \mathbb{R} \rightarrow I
\times \mathbb{R} $ are given by
\begin{align}\label{eq:wnk}
 \om_{n,k}(x,y)=(L_n(x),G_{n,k}(x,y)) \
\mbox{for all} \ (x,y) \in \mathbb{R}^2
\end{align}
By~\eqref{eq:jc}, it is observed that $\ \om_{n,k}$ are continuous
functions.

To introduce a SFIF associated with SIFS~\eqref{eq:sifs},  let
$\{W_k : H(\mathbb{R}^2) \rightarrow H(\mathbb{R}^2),\
k=1,2,\ldots,M \} $,  be a collection of continuous functions
defined by
\begin{align}\label{eq:Wk}
W_k(G) = \bigcup\limits_{n=1}^N \om_{n,k}(G),\  \mbox{where} \
\om_{n,k}(G) = \{\om_{n,k}(x,y) \ \mbox{for all} \ (x,y) \in G\}.
\end{align}
Then, $ \left\{\mathcal{H}(\mathbb{R}^2);\ W_1,\ldots, W_M \right\}
$ is a hyperbolic IFS, since $ h(W_k(A),W_k(B)) \leq \max\limits_{1
\leq n \leq N} \gm_{n,k} \ h(A,B)$, where $h$ is Hausdorff metric on
$\mathcal{H}(\mathbb{R}^2)$. Hence, by Banach fixed point theorem,
there exists an attractor $\mathcal{A} \in
\mathcal{H}(\mathcal{H}(\mathbb{R}^2))$.

\newpage
Let $ \Lambda $ be the code space on $M$ natural
numbers $ {1,2,\ldots,M}$ . In the construction of SFIF, for a
$\sigma =\sigma_1\sigma_2 \ldots \in \Lambda$, let  the action of
SIFS~\eqref{eq:sifs} at the iteration level $j$ be defined by $S_j =
W_{\sigma_j}(S_{j-1})$, where $S_0 $ is the set of given
interpolation data. For a fixed $\sigma \in
\Lambda$, define
\begin{align}\label{eq:Gsigma}
G_{\sigma} \equiv \lim\limits_{k \rightarrow \infty}
W_{\sigma_k} \circ \ldots \circ W_{\sigma_1} (S_0) = \lim\limits_{k
\rightarrow \infty} S_k.
\end{align}
The following
proposition is instrumental for precise definition of a SFIF:

\begin{proposition}\label{prop:SFIF}
Let $G_{\sigma} $ be defined by~\eqref{eq:Gsigma}. Then, $G_{\sigma}
$ is the attractor of SIFS~\eqref{eq:sifs} for $\sigma =
\sigma_1\sigma_2\ldots \sigma_k\ldots\in \Lambda$ and is graph of a
continuous function $g_{\sigma} : I \rightarrow \mathbb{R}$ such
that $ g_{\sigma} (x_n) = y_n$ for all $\ n=0,1,\ldots,N$.
\end{proposition}

Super Fractal Interpolation Function (SFIF) is defined using
Proposition~\ref{prop:SFIF} as follows:
\begin{definition}
The  \textbf{Super Fractal Interpolation Function (SFIF) } for the
given interpolation data $\{(x_i,y_i) :i = 0,1,\ldots,N \}$   is
defined as the function $g_{\sigma}$ whose graph $G_{\sigma}$ is the
attractor of SIFS~\eqref{eq:sifs}.
\end{definition}

\section{Cubic Spline SFIF}\label{sec:convsfif}
The convergence of Cubic Spline SFIF is investigated here using the
conditions of differentiability found in~\cite{kapoor_c2}.
Throughout in this section, the interpolation data $\{(x_n,y_n) :
n=0,1,\ldots,N\}$  is assumed to be such that $x_n - x_{n-1} = h,\ $
for $\ n=1,2,\ldots,N$ and $x_0 =0$. Also, throughout in the sequel,
the SIFS~\eqref{eq:sifs} is chosen such that $\gm_{n,k} =\gm_k ,\
n=1,2,\ldots,N$, for some $\gm_k, \ 0<\gm_k<1,\ k \in
\{1,2,\ldots,M\}$.

\begin{definition} A SFIF $g_{\sigma}$, associated with SIFS~\eqref{eq:sifs}, is called
\textbf{Cubic Spline SFIF} if $q_{n,k} $, given in~\eqref{eq:gnk},
are cubic polynomials for $ n=1,2,\ldots,N$ and $k = 1,2,\ldots,M $.
\end{definition}

It is observed that, if $q_{n,k}(x) = q_{n,k,3} x^3 + q_{n,k,2} x^2
+ q_{n,k,1} x + q_{n,k,0} $, the coefficients $q_{n,k,i},\ i=0,1,2,3
$,  depend upon $\gm_k$ due to~\eqref{eq:jc}, necessitating in the
sequel, the use of notation $q_n(\gm_k,x)$  in place of
$q_{n,k}(x)$. Throughout in this section, it is assumed that, for
some $A_0 \geq 0$, the polynomials  $q_n$ satisfy
\begin{align}\label{eq:cndsfif1}
 \frac{|q_n(\gm_k,x)-q_n(\gm_l,x)|}{|\gm_k
-\gm_l|} \leq A_0,
\end{align}
for $n=1,2,\ldots,N, \ k,l=1,2,\ldots,M$ and $x \in [x_0,x_N]$.

Let $\gm_{k_0}$  be such that $ |\gm_{k_0}| < 1$ ,  $\bt_{k_0} =
\max\limits_{1 \leq l \leq M}{ |\gm_l-\gm_{k_0}| } < 1$ and
$\varsigma = \varsigma_1\varsigma_2\ldots\varsigma_j\ldots \in
\Lambda$ be such that $\varsigma_j = k_0$ for all $j \in
\mathbb{N}$. Consider the family of continuous functions
\begin{align}\label{eq:Grepeat}
 {\cal
G} = \{ f : I \rightarrow \mathbb{R} \ \mbox{such that} \ f \ \mbox{
is continuous}, f(x_0) = y_0 \ \mbox{and} \ f(x_N) = y_N\}
\end{align}
with metric  $ d_{{\cal G}} (f,g) = \max\limits_{x \in I } |f(x) -
g(x)| $.  For Read-Bajraktarevic operator $T : \Lambda \times {\cal
G} \rightarrow {\cal G}$
 defined by
\begin{align}\label{eq:rb}
T(\sigma, g)(x) & =  \lim\limits_{k \rightarrow \infty}
\bigg\{G_{i_k,\sigma_{k}} \bigg( L_{i_k}^{-1} (x),
G_{i_{k-1},\sigma_{k-1}} \Big( L_{i_{k-1}}^{-1}\circ L_{i_k}^{-1}
(x), G_{i_{k-2},\sigma_{k-2}}\big(. ,\ldots \nno \\ & \quad \quad
 \quad \quad G_{i_1,\sigma_1}(L_{i_1}^{-1}\circ \ldots \circ
L_{i_k}^{-1}(x),g(L_{i_1}^{-1}\circ \ldots \circ L_{i_k}^{-1}(x)))
\ldots \big)\Big)\bigg)\bigg\}
\end{align}
where, $\Lambda$ is the code space on $M$ natural numbers $1,2,
\ldots,M$  and ${\cal G}$ is given by~\eqref{eq:Grepeat},  the
following proposition gives a bound on
$\|T(\sigma,g)(x)-T(\varsigma,g) (x)\| $ for $\sigma, \varsigma \in
\Lambda$:

\begin{proposition}\label{prop:sfif1}
Let $g \in {\cal G}$  and inequality~\eqref{eq:cndsfif1} be
satisfied. Then,
\begin{align}\label{eq:sfif1}
\|T(\sigma,g)-T(\varsigma,g)\|_{\infty}  & \leq  \bt_{k_0} \ \left\{
\|g\|_{\infty} +   \frac{ A_0}{1 - \gm_*} +\frac{B_0}{1 - \bt_{k_0}}
\right\}
\end{align}
where $ \sigma, \varsigma \in \Lambda$,  $\ \max\limits_{\substack{
 x \in [x_0,x_N] \\ n=1,2,\ldots,N} }|q_{n}(\gm_{k_0},x)| = B_0 $,  $ \
\gm_* = \max\limits_{1 \leq l \leq M}{ |\gm_l| } < 1 \ $ and  $
\bt_{k_0} = \max\limits_{1 \leq l \leq M}{ |\gm_l-\gm_{k_0}| }~<~1$.
\end{proposition}

\begin{proof}
By the definition of $T(\sigma,g)(x)$ (c.f.~\eqref{eq:rb}),
\begin{align}\label{eq:Tg1}
\lefteqn{|T(\sigma,g)(x)-T(\varsigma,g) (x) |} \nno \\ & \leq
\lim\limits_{k \rightarrow \infty} \Bigg\{
\left(\prod\limits_{j=1}^k |\gm_{\sigma_j} - \gm_{k_0}| \right) |g
\big(L_{i_1}^{-1} \circ \ldots \circ L_{i_k}^{-1} (x)
\big)| \nno \\
& \quad \mbox{} + \sum\limits_{p=1}^k \left(\prod\limits_{j=p+1}^k
|\gm_{\sigma_j} | \right) \left|q_{i_p}
\big(\gm_{\sigma_p},L_{i_p}^{-1} \circ \ldots \circ L_{i_k}^{-1} (x)
\big) - q_{i_p}\big(\gm_{k_0},L_{i_p}^{-1}
\circ \ldots \circ L_{i_k}^{-1} (x) \big)\right| \nno \\
& \quad \mbox{} + \sum\limits_{p=1}^{k-1}
\left(\prod\limits_{j=p+1}^k |\gm_{\sigma_j} - \gm_{k_0}| \right)  \
|q_{i_p} \big(\gm_{k_0}, L_{i_p}^{-1} \circ \ldots \circ
L_{i_k}^{-1} (x) \big)| \Bigg\}.
\end{align}

Since $q_{n}(\gm_{k_0},x)$ are cubic polynomials defined on compact
set $[x_0, x_N]$, there exists a $B_0 > 0$  such that
$\max\limits_{\substack{ x \in [x_0,x_N] \\ n=1,2,\ldots,N}
}|q_{n}(\gm_{k_0},x)| = B_0 $. Therefore, by~\eqref{eq:cndsfif1}
and~\eqref{eq:Tg1}, it follows that
\begin{align*}
|T(\sigma,g)(x)-T(\varsigma,g) (x) | & \leq \lim\limits_{k
\rightarrow \infty} \Bigg\{\left(\prod\limits_{j=1}^k
|\gm_{\sigma_j} - \gm_{k_0}| \right) \|g\|_{\infty} \\
& \ \mbox{} + \sum\limits_{p=1}^k \left(\prod\limits_{j=p+1}^k
|\gm_{\sigma_j} | \right)\ |\gm_{\sigma_p} - \gm_{k_0}|\ A_0   +
\sum\limits_{p=1}^{k-1} \left(\prod\limits_{j=p+1}^{k-1}
|\gm_{\sigma_j} - \gm_{k_0}| \right) \ B_0  \Bigg\}.
\end{align*}

Since $\bt_{k_0} = \max\limits_{1 \leq l \leq M}{ |\gm_l-\gm_{k_0}|
} < 1$ and $ \gm_* = \max\limits_{1 \leq l \leq M}{ |\gm_l| } < 1
$,~\eqref{eq:sfif1} follows from the above inequality. \qed
\end{proof}

The following proposition gives a bound on $\|g_{\sigma} -
g_{\varsigma}\| $,  $\sigma, \varsigma \in \Lambda$:

\begin{proposition}\label{prop:sfif2}
Let $g_{\sigma}, g_{\varsigma} \in {\cal G}$ and~\eqref{eq:cndsfif1}
be satisfied. Then,
\begin{align}\label{eq:sfif2}
\|g_{\sigma}-g_{\varsigma}\|_{\infty} \leq \frac{\bt_{k_0}}{(1-
 \gm_*)} \left( \|g_{\varsigma}\|_{\infty} +   \frac{  \ A_0}{1 - \gm_*} +\frac{
 B_0}{1 - \bt_{k_0}} \right)
\end{align}
where $ \gm_*, \bt_{k_0} $ and $B_0$ are as in
Proposition~\ref{prop:sfif1}.
\end{proposition}

\newpage
\begin{proof} Since~\eqref{eq:cndsfif1} is satisfied,
$g_{\sigma} = T (\sigma,g_{\sigma})$ and  $ g_{\varsigma} = T
(\varsigma,g_{\varsigma})$ for Read-Bajraktarevic operator $T$,
defined by~\eqref{eq:rb},
\begin{align*}
\|g_{\sigma}- g_{\varsigma}\|_{\infty} & \leq \gm_* \ \|g_{\sigma}-
g_{\varsigma}\|_{\infty}  + \bt_{k_0} \ \left(
\|g_{\varsigma}\|_{\infty} +   \frac{  \ A_0}{1 - \gm_*} +\frac{
 B_0}{1 - \bt_{k_0}} \right).
\end{align*}
The inequality~\eqref{eq:sfif2} now follows from the above
inequality. \qed \end{proof}

\begin{remark}
For $\gm_{k_0} = 0$, $\gm_* = \bt_{k_0}$.  Thus,
inequality~\eqref{eq:sfif2} for $\gm_{k_0} = 0$ implies
\begin{align*}
\|g_{\sigma}-g_{\varsigma}\|_{\infty} & \leq \frac{\gm_*}{1-\gm_*}
\left(  \ \|g_{\varsigma}\|_{\infty} +
 \frac{A_0 +\ B_0}{1-\gm_*}
\right).
\end{align*}

 By Hall and Meyer's theorem~\cite{hall76},  $\|g_{\varsigma}\|_{\infty} \leq
 K_0 h^4 + J_0$. Consequently,  Proposition~\ref{prop:sfif2}, for $\gm_{k_0} = 0$
gives,
\begin{align}\label{eq:sfif3}
\|g_{\sigma}-g_{\varsigma}\|_{\infty} & \leq \frac{\gm_*}{1-\gm_*}
\bigg(  K_0 h^4 + J_0 +
 \frac{ A_0 + B_0}{1-\gm_*}
\bigg).
\end{align}
\end{remark}

Using inequality~\eqref{eq:sfif3}, the order of approximation of
data generating function $y(x)$ by SFIF $g_{\sigma}$ is given by the
following theorem:

\begin{theorem}\label{th:ero}
Let $y(x) \in C^4[x_0,x_N] $ be a data generating function and
$g_{\sigma} \in {\cal G}$ be a SFIF associated with
SIFS~\eqref{eq:sifs} such that $\gm_*(h) =
\max\limits_{i=1,2,\ldots,N}{ |\gm_i| } \leq
 \frac{h^{2+s}}{|I|^{2+s}}$, for  some $s,\ 0 < s < 1$, where $\ h = |x_i - x_{i-1}| ,\ i=1,2,\ldots,N \ $ and $\ |I| = |x_N - x_0|$.  Then, for $ \ 0 < \epsilon < s$,
\begin{align} \label{eq:ero}
\|y-g_{\sigma}\|_{\infty} = {\small o}(h^{2+\epsilon}).
\end{align}
provided~\eqref{eq:cndsfif1} holds.
\end{theorem}

\begin{proof}  Since~\eqref{eq:cndsfif1} holds,
an application of inequality~\eqref{eq:sfif3} gives,
\begin{align}\label{eq:gb}
\|y-g_{\sigma}\|_{\infty} & \leq \|y - g_{\varsigma}\| +
\|g_{\varsigma}-g_{\sigma}\|_{\infty} \nno  \\
& \leq K_0 h^4   + \frac{\gm_*(h)}{1-\gm_*(h)} \bigg(  K_0 h^4 + J_0
+  \frac{ A_0 + B_0}{1-\gm_*(h)} \bigg) \nno  \\ & \leq
\frac{1}{1-\gm_*(h)}  \bigg( K_0 h^4  +  \gm_*(h) J_0   +
 \frac{\gm_*(h)  ( A_0 +\ B_0)}{1-\gm_*(h)} \bigg).
\end{align}

Using $ | \gm_*(h) | \leq \frac{h^{2+s}}{T^{2+s}}$, the
inequality~\eqref{eq:gb} implies
\begin{align}\label{eq:errbnd}
\|y-g_{\sigma}\|_{\infty}   &  \leq \frac{|I|^{(2+s)}}{|I|^{(2+s)} -
h^{(2+s)}} \Bigg\{ K_0 h^4 +  \left( \frac{ J_0 \
h^{(2+s)}}{|I|^{(2+s)} }\right) + \left( \frac{ ( A_0 + B_0 )\
h^{(2+s)}}{|I|^{(2+s)} - h^{(2+s)}}\right) \Bigg\}.
\end{align}

The order of approximation error given by~\eqref{eq:ero} follows
from the above inequality. \qed \end{proof}

\begin{remark}
It follows from inequality~\eqref{eq:errbnd} that, in fact,  $ \
\|y-g_{\sigma}\|_{\infty} = O(h^{2+s})$.
\end{remark}

\begin{remark}
If $M=1$ in SIFS~\eqref{eq:sifs}, then $g_{\sigma}$ reduces to a
FIF. The convergence result for a Cubic Spline
FIF~\cite{navascues03} follows as a particular case of
Theorem~\ref{th:ero}.
\end{remark}

The order in approximation of derivatives of  data generating
function by  corresponding derivatives of  SFIF is now investigated.
It is known~\cite{kapoor_c2} that $ g_{\sigma}^{(1)}(x) $ and  $
g_{\sigma}^{(2)} ( x ) $  are SFIFs associated with SIFSs $\Big\{
\big\{\mathbb{R}^2;\ \om_{i,k,j}(x,y) = (L_i(x),G_{i,k,j}(x,y)) : i
= 1,2,\ldots,N\big\} , k = 1,2,\ldots,M \Big\}$ for  $ j =1,2$
respectively. Here, the functions $ G_{i,k,1}(x,y)$ and $
G_{i,k,2}(x,y)$  are given by \vspace{-0.5cm}
\begin{align*}
 G_{i,k,1}(x,y) = N \gm_k y + N
q_i^{(1)}(\gm_k,x) \quad \mbox{and} \quad G_{i,k,2}(x,y)  = N^2
\gm_k y + N^2 q_i^{(2)}(\gm_k,x).
\end{align*}
Let, for some $A_j \geq 0$, the
polynomials $q_n$ satisfy \vspace{-0.5cm}
\begin{align}\label{eq:cndsfif2}
 \frac{ |q_n^{(j)}(\gm_k,x)-q_n^{(j)}(\gm_l,x)|}{|\gm_{k} -\gm_{l}|\ }
\leq A_j, \quad  j=0,1,2,
\end{align}
for all $n=1,2,\ldots,N, \ k,l=1,2,\ldots,M$ and $x \in [x_0,x_N]$.

 For $j=1,2$, define the Read-Bajraktarevic operator $T_j: \Lambda \times {\cal
 G} \rightarrow {\cal G}$ by
\begin{align}\label{eq:Tp}
T_j(\sigma,g)(x) & = \lim\limits_{k \rightarrow \infty}
\bigg\{G_{i_k,\sigma_k,j}\bigg(L_{i_k}^{-1}(x),
G_{i_{k-1},\sigma_{k-1},j}\Big(L_{i_{k-1}}^{-1} \circ
L_{i_k}^{-1}(x), G_{i_{k-2},\sigma_{k-2},j} \Big( .,\ldots \nno \\
& \quad \quad G_{i_1,\sigma_1,j}(L_{i_1}^{-1} \circ \ldots \circ
\ldots L_{i_k}^{-1}(x), g(L_{i_1}^{-1} \circ \ldots \circ \ldots
L_{i_k}^{-1}(x)))\Big) \Big) \bigg) \bigg\}
\end{align}
where, $\Lambda$ is the code space on $M$ natural numbers $1,2,
\ldots,M$  and ${\cal G}$ is given by~\eqref{eq:Grepeat}. To find
the order of approximation of derivatives of data generating
function $y(x)$ by corresponding derivatives of SFIF $g_{\sigma}$,
the bounds on $\|T_j(\sigma,g)(x)-T_j(\varsigma,g) (x)\|$ similar
to~\eqref{eq:sfif1} and the bounds on $\|g_{\sigma}^{(j)} -
g_{\varsigma}^{(j)}\| , j=1,2,$ similar to~\eqref{eq:sfif2} are
needed. Such a bound on $\|T_j(\sigma,g)(x)-T_j(\varsigma,g) (x)\| $
for $\sigma, \varsigma \in \Lambda$, is given  by the following
proposition:

\begin{proposition}\label{prop:sfifd1}
Let $g \in {\cal G}$  and inequality~\eqref{eq:cndsfif2} be
satisfied. Then, for $j=1,2$,
\begin{align}\label{eq:sfifd1}
\|T_j(\sigma,g)-T_j(\varsigma,g)\|  & \leq  N^j \bt_{k_0} \ \Bigg\{
\|g\|_{\infty}
 +    \frac{ A_j}{1 - N^j \gm_*} + \frac{N^j B_j}{1 - N^j \bt_{k_0}}
 \Bigg\}
\end{align}
where, $\sigma, \varsigma \in \Lambda$, $\ \max\limits_{\substack{ x
\in [x_0,x_N] \\n=1,2,\ldots,N}} |q_{n}^{(j)}(\gm_{k_0},x)| \leq B_j
$, $\  \gm_* = \max\limits_{1 \leq l \leq M}{ |\gm_l| } <
\frac{1}{N^2} \\ $  and $\ \bt_{k_0} = \max\limits_{1 \leq l \leq M}
{|\gm_l - \gm_{k_0}|} <~\frac{1}{N^2}$.
\end{proposition}

\begin{proof}
 By~\eqref{eq:Tp},
\begin{align}\label{eq:Tgd1}
\lefteqn{|T_j(\sigma,g)(x)-T_j(\varsigma,g) (x) |}\nno \\  &  \leq
\lim\limits_{k \rightarrow \infty} \bigg\{N^{j k}
\left(\prod\limits_{n=1}^k |\gm_{\sigma_n} - \gm_{k_0}| \right) |g
\big(L_{i_1}^{-1} \circ \ldots \circ L_{i_k}^{-1} (x) \big)|   \nno \\
& \quad \mbox{} + \sum\limits_{m=1}^k N^{j(k-m+1)}
\left(\prod\limits_{n=m+1}^k |\gm_{\sigma_n} | \right) \times \nno \\
& \quad \quad \mbox{} \times |q_{i_m}^{(j)}
\big(\gm_{\sigma_m},L_{i_m}^{-1} \circ \ldots \circ L_{i_k}^{-1} (x)
\big) - q_{i_m}^{(j)}
\big(\gm_{k_0},L_{i_m}^{-1} \circ \ldots \circ L_{i_k}^{-1} (x) \big)|   \nno \\
& \quad \mbox{} + \sum\limits_{m=1}^{k-1} N^{j (k-m+1)}
\left(\prod\limits_{n=m+1}^k |\gm_{\sigma_n} - \gm_{k_0}| \right)  \
|q_{i_m}^{(j)} \big(\gm_{k_0},L_{i_m}^{-1} \circ \ldots \circ
L_{i_k}^{-1} (x) \big)| \bigg\}.
\end{align}

Since $q_{n}^{(j)}(\gm_{k_0},x)$ are polynomials defined on compact
set $[x_0, x_N]$, there exists a $B_j > 0$  such that
$\max\limits_{\substack{ x \in [x_0,x_N] \\n=1,2,\ldots,N}}
|q_{n}^{(j)}(\gm_{k_0},x)| \leq B_j $. Therefore,
by~\eqref{eq:cndsfif2} and~\eqref{eq:Tgd1}, it follows that
\begin{align*}
|T_j(\sigma,g)(x)-T_j(\varsigma,g) (x) | & \leq \lim\limits_{k
\rightarrow \infty} \Bigg\{ N^{j k} \left(\prod\limits_{n=1}^k
|\gm_{\sigma_n} - \gm_{k_0}| \right)
\|g\|_{\infty} \\
& \quad \mbox{} + \sum\limits_{m=1}^k  N^{j(k-m+1)}
\left(\prod\limits_{n=m+1}^k |\gm_{\sigma_n} | \right)\
|\gm^{\sigma_m} - \gm_{k_0}|\ A_j  \\
& \quad  \mbox{} + \sum\limits_{m=1}^{k-1} N^{j(k-m+1)}
\left(\prod\limits_{n=m+1}^k |\gm_{\sigma_n} - \gm_{k_0}| \right)
B_j \bigg\}.
\end{align*}

Since $\bt_{k_0} = \max\limits_{1 \leq l \leq M}( |\gm_l-\bt| )  <
\frac{1}{N^2}\ $ and  $ \ \gm_* = \max\limits_{1 \leq l \leq M}
|\gm_l|   < \frac{1}{N^2}$,~\eqref{eq:sfifd1} follows from the above
inequality. \qed \end{proof}

Similar to Proposition~\ref{prop:sfif2}, the following proposition
gives a bound on $\|g_{\sigma}^{(j)} - g_{\varsigma}^{(j)}\| $ for
$\sigma, \varsigma \in \Lambda$:

\begin{proposition}\label{prop:sfifd2}
 Let $g_{\sigma}, g_{\varsigma} \in
{\cal G}$ and~\eqref{eq:cndsfif2} be satisfied. Then, for $j=1,2$,
\begin{align}\label{eq:sfifd2}
\|g_{\sigma}^{(j)}-g_{\varsigma}^{(j)}\|_{\infty} &\leq \frac{N^j
\bt_{k_0}}{1- N^{j} \gm_*} \Bigg\{ \|g_{\varsigma}^{(j)}\|_{\infty}
+ \frac{ A_j}{1 - N^j \gm_*} + \frac{N^j B_j}{1 - N^j \bt_{k_0}}
\Bigg\}
\end{align}
where, $ \gm_*, \bt_{k_0} $ and $B_j$ are as in
Proposition~\ref{prop:sfifd1}.
\end{proposition}

\begin{proof} Since~\eqref{eq:cndsfif2} is satisfied,
$g_{\sigma}^{(j)}  = T_j (\sigma,g_{\sigma}^{(j)})$ and
$g_{\varsigma}^{(j)} = T_j (\varsigma,g_{\varsigma}^{(j)})$ for
Read-Bajraktarevic operator $T_j$, defined by~\eqref{eq:Tp},
\begin{align*}
\|g_{\sigma}^{(j)}- g_{\varsigma}^{(j)} \|_{\infty} & \leq N^j \gm_*
\|g_{\sigma}^{(j)}- g_{\varsigma}^{(j)} \|_{\infty} +  N^j \bt_{k_0}
\ \Bigg\{ \|g_{\varsigma}^{(j)}\|_{\infty}
 +    \frac{ A_j}{1 - N^j \gm_*} + \frac{N^j B_j}{1 - N^j \bt_{k_0}}
 \Bigg\}.
\end{align*}
The inequality~\eqref{eq:sfifd2} now follows from the above
inequality. \qed \end{proof}

\begin{remark}
For $\gm_{k_0} = 0$, $\gm_* = \bt_{k_0}$.  Thus,
inequality~\eqref{eq:sfifd2} implies
\begin{align*}
\|g_{\sigma}^{(j)}-g_{\varsigma}^{(j)}\|_{\infty} & \leq \frac{N^j
\gm_*}{1- N^{j} \gm_*} \Bigg\{ \|g_{\varsigma}^{(j)}\|_{\infty} +
\frac{ A_j + N^j B_j}{1 - N^j \gm_*}  \Bigg\}.
\end{align*}
By Hall and Meyer's theorem~\cite{hall76}, $\
 \|g_{\varsigma}^{(j)}\|_{\infty} \leq
 K_j h^4 + J_j$.
Consequently,  Proposition~\ref{prop:sfifd2}, for $\gm_{k_0}=0$,
gives,
\begin{align}\label{eq:sfifd3}
\|g_{\sigma}^{(j)}-g_{\varsigma}^{(j)}\|_{\infty} & \leq \frac{N^j
\gm_*}{1- N^{j} \gm_*} \Bigg\{ K_j h^4 + J_j + \frac{ A_j + N^j
B_j}{1 - N^j \gm_*}  \Bigg\}.
\end{align}
\end{remark}

Using inequality~\eqref{eq:sfifd3}, the orders of approximation of
derivatives of data generating function $y(x)$ by corresponding
derivatives of SFIF $g_{\sigma}$ are given by the following theorem:

\begin{theorem}\label{th:dero}
Let $y(x) \in C^4[x_0,x_N] $ be a data generating function and
$g_{\sigma} \in {\cal G}$ be a SFIF associated with
SIFS~\eqref{eq:sifs} such that $\gm_*(h) =
\max\limits_{i=1,2,\ldots,N}{ |\gm_i| } \leq
 \frac{h^{2+s}}{|I|^{2+s}}$, for some $s, \ 0 < s < 1$, where
  $\ h = x_i - x_{i-1},\ i=1,2,\ldots,N \ $  and $\ |I| = x_N - x_0$.
   Then, for $j=1,2$  and $ 0 <
\epsilon < s$,
\begin{align} \label{eq:dero}
  \|y^{(j)}-g_{\sigma}^{(j)}\|_{\infty} = {\small
o}(h^{2-j+\epsilon})
\end{align}
provided~\eqref{eq:cndsfif2} holds.
\end{theorem}

\begin{proof}
Since~\eqref{eq:cndsfif2} holds, an application of
inequality~\eqref{eq:sfifd3} gives,
\begin{align}\label{eq:gdb}
\lefteqn{\|y^{(j)}-g_{\sigma}^{(j)}\|_{\infty} } \nno \\ &  \leq
\|y^{(j)}-g_{\varsigma}^{(j)}\|
+ \|g_{\varsigma}^{(j)}-g_{\sigma}^{(j)}\|_{\infty} \nno \\
&  \leq K_j \ h^{4-j} +  \frac{1}{1- N^{j} \gm_*(h)} \Bigg\{  N^j
\gm_*(h) \left( K_j h^{4-j} +J_j \right)  + \frac{N^j \gm_*(h) } {1
- N^j \gm_*(h)} \left[  A_j + N^j B_j   \right]
   \Bigg\}  \nno \\
&  \leq \frac{1}{1- N^{j} \gm_*(h)} \bigg\{ K_j\ h^{4-j}  + N^{j}\
\gm_*(h) \ J_j    +  \frac{N^j \gm_*(h) } {1 - N^j \gm_*(h)} \left[
A_j + N^j B_j \right]   \Bigg\}.
\end{align}
Using $ | \gm_*(h) | \leq \frac{h^{2+s}}{|I|^{2+s}}$, the
inequality~\eqref{eq:gdb} implies
\begin{align}\label{eq:errbndderi}
\lefteqn{\|y^{(j)}-g_{\sigma}^{(j)}\|_{\infty}} \nno \\ & \leq
\frac{|I|^{(2+s-j)} }{|I|^{(2+s-j)} - h^{(2+s-j)} }\bigg\{ K_j\
h^{4-j}
  + \frac{J_j  \ h^{(2+s-j)}}{|I|^{(2+s-j)} } + \frac{h^{(2+s-j)} }{|I|^{(2+s-j)} - h^{(2+s-j)}}
\left[  A_j + N^j B_j   \right]
  \Bigg\}.
\end{align}
The order of approximation error given by~\eqref{eq:dero} follows
from the above inequality. \qed \end{proof}

\begin{remark}
It follows from inequality~\eqref{eq:errbndderi} that, in fact, $
\|y^{(j)}-g_{\sigma}^{(j)}\|_{\infty} = O(h^{2+s-j})$,   $j=1,2$.
\end{remark}

\begin{remark}
If $M=1$ in SIFS $\Big\{ \big\{\mathbb{R}^2;\ \om_{i,k,j}(x,y) =
(L_i(x),G_{i,k,j}(x,y)) : i = 1,2,\ldots,N\big\} , k = 1,2,\ldots,M
\Big\}$, $j=1,2$, then $g_{\sigma}^{(j)}$, $j=1,2$, are FIFs. The
convergence results for derivatives of a Cubic Spline
FIFs~\cite{navascues03} follow as a particular case of
Theorem~\ref{th:dero}.
\end{remark}

\section{Conclusions}
In the present work,  the notion of Cubic Spline SFIF is introduced
for an efficient approximation of the data generating function. The
approximation properties of a Cubic Spline SFIF are investigated and
it is proved that the order of approximation of the data generating
function $y(x) \equiv y^{(0)}(x)$ and its derivatives $y^{(j)}(x)$
by an interpolating Cubic Spline SFIF $g_{\sigma} \equiv
g_{\sigma}^{(0)}$ and its corresponding derivatives
$g_{\sigma}^{(j)}$ respectively, is $o(h^{2-j+\epsilon}),\
0<\epsilon<1 ,\ j=0,1,2 $. These convergence results show that it is
possible to approximate any regular data generating function by
Cubic Spline SFIF with arbitrary accuracy. Our study of Cubic Spline
SFIF is likely to have wide applications like pattern-forming alloy
solidification in chemistry, blood vessel patterns in biology,
signal processing, fragmentation of thin plates in engineering,
stock markets in finance, wherein significant randomness and
variability is observed in simulation of various processes.

\section*{Acknowledgments}  The author Srijanani thanks CSIR for
research grant (No:9/92(417)/2005-EMR-I) for the present work.

\bibliographystyle{unsrt}

\begin{thebibliography}{10}

\bibitem{barnsley86}
{Barnsley M.F.}
\newblock Fractal functions and interpolation.
\newblock {\em Constructive Approximation}, 2:303--329,~1986.

\bibitem{barnsley05}
{Barnsley M.F.}, {Hutchinson J.E.}, and {Stenflo O}.
\newblock A fractal valued random iteration algorithm and fractal hierarchy.
\newblock {\em Fractals}, 13(2):111--146,~2005.

\bibitem{barnsley06}
{Barnsley M.F.}
\newblock {\em Super Fractals}.
\newblock Cambridge University Press,~2006.

\bibitem{barnsley08}
{Barnsley M.F.}, {Hutchinson J.E.}, and {Stenflo O.}
\newblock V-variable fractals: Fractals with partial self similarity.
\newblock {\em Advances in Mathematics}, 218:2051--2088,~2008.

\bibitem{massopust10}
{Massopust P.}
\newblock {\em Interpolation and Approximation with Splines and Fractals}.
\newblock Oxford University Press,~2010.

\bibitem{kapoor_c2}
{Kapoor G.P.} and {Prasad S.A.}
\newblock Super fractal interpolation functions.

\bibitem{schoenberg46}
{Schoenberg I.J.}
\newblock Contributions to the problem of approximation of equidistant data by
  analytic functions.
\newblock {\em Quartely Applied Mathematics}, 4:45--99, 1946.

\bibitem{knott00}
{Knott G.D.}
\newblock {\em Interpolationg Cubic Splines}.
\newblock Birkhauser, Boston,~2000.

\bibitem{bartels89}
{Bartels R.H.}, {Bealty J.C.}, {Beatty J.C.}, and {Barsky B.A.}
\newblock {\em An introduction to Splines for Use in Computer Graphics and
  Geometric Modelling}.
\newblock Morgan Kauffmann Publishers, San Mateo,~1989.

\bibitem{farin90}
{Farin G.}
\newblock {\em Curves and Surfaces for Computer Aided Geometric Design: A
  Practical Guide}.
\newblock Academic Press, San Diego,~1990.

\bibitem{micula99}
{Micula G.} and {Micula S.}
\newblock {\em Handbook of Splines}.
\newblock Kluwer Academic Publishers, Dordrecht,~1999.

\bibitem{navascues03}
{Navascues M.A.} and {Sebastian M.V.}
\newblock Some results of convergence of cubic spline fractal interpolation
  functions.
\newblock {\em Fractals}, 11(1):1--7,~2003.

\bibitem{hall76}
{Hall C.A.} and {Meyer W.W.}
\newblock Optimal error bounds for cubic spline interpolation.
\newblock {\em Journal of Approximation Theory}, 16:105--122,~1976.

\end{thebibliography}

\end{document}